\documentclass[12pt]{article}
\usepackage{amsmath,amssymb,amsbsy,amsfonts,amsthm,latexsym,
            amsopn,amstext,amsxtra,euscript,amscd,amsthm,url}
\usepackage[margin=3cm]{geometry}
\usepackage[usenames, dvipsnames]{color}
\usepackage[utf8]{inputenc}
\usepackage[english]{babel}
\usepackage{mathrsfs}
\usepackage{caption}
\usepackage{xcolor}

\theoremstyle{plain}
\newtheorem{theorem}{Theorem}[section]

\newtheorem{lemma}[theorem]{Lemma}
\newtheorem{corollary}[theorem]{Corollary}

\theoremstyle{definition}

\theoremstyle{remark}
\newtheorem{remark}[theorem]{Remark}

\newcommand{\icomp}{\texttt{i}}

\DeclareMathOperator{\1}{\textbf{1}}
\newcommand{\Ocal}{O}

\DeclareMathOperator{\id}{id}
\global\long\def\epsilon{\varepsilon}

\begin{document}
\date{}

\title{Sums of weighted averages of gcd-sum functions {II}}
\author{Isao Kiuchi and Sumaia Saad Eddin}

\maketitle
{\def\thefootnote{}
\footnote{{\it Mathematics Subject Classification 2010: 11A25, 11N37, 11Y60.}}

\begin{abstract}
In this paper, we establish the following two identities involving the Gamma function and Bernoulli polynomials, namely    
$$
\sum_{k\leq x}\frac{1}{k^s}  \sum_{j=1}^{k^s}\log\Gamma\left(\frac{j}{k^s}\right)
                              \sum_{\substack{d|k \\ d^{s}|j}}f*\mu(d)
\quad 
{\rm and }
\quad   
\sum_{k\leq x}\frac{1}{k^s}\sum_{j=0}^{k^{s}-1} 
 B_{m}\sum_{\substack{d|k \\ d^{s}|j}} f*\mu(d) 
$$
with any fixed  integer $s> 1$ and  any arithmetical function $f$.  
We give asymptotic formulas for them with various multiplicative functions $f$. 
We also consider several formulas of Dirichlet series associated with the above identities. This paper is a continuation of an earlier work of the authors. 
\end{abstract}

\maketitle

\section{Introduction}
Throughout the paper we use the following notations: Let $\mathbb{N}$ be the set of positive integers and $\mathbb{N}_0=\mathbb{N} \cup\{0\}$. 
Let $f$ and $g$ be two arithmetical functions. 
The Dirichlet convolution of $f$ and $g$ is defined by  $f * g(n) = \sum_{d \mid n} f(d) g\left({n}/{d}\right)$  
for $n\in \mathbb{N}$.
The functions $\mu, \phi$ and  $\psi$, as usual, denote the M\"obius function, 
the Euler totient function and the Dedekind function.
We define arithmetic functions $\1$ and $\id$ by $\1(n) = 1$ 
and $\id (n)=n$ for all $n$. 
We recall that Bernoulli polynomials are defined by the generating function
\begin{equation*}
\frac{t e^{x t}}{e^{t}-1}=\sum_{n=0}^{\infty}B_n(x)\frac{t^n}{n!},
\end{equation*}
for $ |t|<2\pi$. For $x=0$, the numbers $B_n=B_n(0)$  are always called Bernoulli numbers. Other notations will be given in the next section.\\

For $j,k \in \mathbb{N}$, let $\gcd (j, k)$ denote their greatest common divisor. For a fixed integer $s\in \mathbb{N}$, let $(a,b)_s$ denote the greatest common s-power divisor of $a$ and $b$. If $a=j$ and $b=k^s,$ then it is 
$$(j, k^s)_s=\max \{d^s\ :\ d^s|j, d|k\}.$$
 Therefore, $(j,k^1)_1=\gcd(j,k)$.
The Ramanujan sum $c_{k}$ is an arithmetic function which is defined as
$$ 
c_{k}(j)=\sum_{\substack{m=1\\ \gcd(m,k)=1}}^{k}\exp \left( 2\pi \icomp m j/k\right)=\sum_{d|\gcd(j,k)}d\mu\left(\frac{k}{d}\right),  
$$  
where $\icomp=\sqrt{-1}$. This function was first introduced by Ramanujan in 1918. In more recent years, various generalizations of the Ramanujan sum have been constructed. 
One of the most known generalizations of $c_{k}$ is according to Cohen \cite{C1},\cite{C2},\cite{C3}, who defined the arithmetic function $c_k^{(s)}$ for $k, j \in \mathbb{N}$ and $s\in \mathbb{N}$ by
\begin{equation*}
c_k^{(s)}(j)= \sum_{\substack{m=1\\ (j,k^s)_s=1}}^{k}\exp \left( 2\pi \icomp m j/k^s\right).
\end{equation*}
The author proved that
\begin{align*}
c_{k}^{(s)}(j)& := \sum_{d^s |(j,k^{s})_{s}}d^{s}\mu\left(\frac{k}{d}\right).  
\end{align*} 
In 1952, Anderson and  Apostol \cite{AA} 
introduced more general arithmetic function $s_{k}$ defined by the identity 
\begin{equation}
\label{eq2}
s_{k}(j)=\sum_{d|\gcd(k,j)}f(d)g\left(\frac{k}{d}\right),
\end{equation}
with any arithmetical functions $f$ and $g$. 
In case that $f=\id$ and $g=\mu$, the formula~\eqref{eq2} becomes the Ramanujan sum $c_k$.
In the same paper, the authors gave several analytic and algebraic properties of $s_{k}$.
This latter is generalized to
\begin{equation}                                                                 \label{ASDFG}
s_{k}^{(s)}(j) := \sum_{d^s |(j,k^{s})_{s}}f(d)g\left(\frac{k}{d}\right).   
\end{equation}
Again, when $f=\id$ and $g=\mu$, the above formula gives the Cohen-Ramanujan sum $c_k^{(s)}(j)$. 
There has been a good amount of work to study both functions $s_{k}$ and $s_k^{(s)}$. Recently, the first author derived identities for the partial sum of weighted averages of $s_{k}(j)$ and $s_k^{(s)}(j)$ with weights being logarithms, Gamma function $\Gamma$, Bernoulli polynomials and others in~\cite{K} and \cite{K2}. For any  real number $x >1$ and any fixed  integers $r, s, m \geq 1$, he proved that:
\begin{multline}     
\label{kiuchi-gamma1}
\sum_{k\leq x}\frac{1}{k^{s(r+1)}} \sum_{j=1}^{k^s}j^r s_{k}^{(s)}(j)
=\frac{1}{2} \sum_{d\ell\leq x}\frac{f(d)}{d^s}\frac{g(\ell)}{\ell^s} 
 +\frac{1}{r+1}\sum_{d\ell\leq x}\frac{f(d)}{d^s}g(\ell)
          \\ +\frac{1}{r+1} \sum_{m=1}^{[r/2]}\binom{r+1}{2m}B_{2m}\sum_{d\ell\leq x}\frac{f(d)}{d^s}\frac{g(\ell)}{\ell^{2ms}},
\end{multline}
\begin{multline}     
\label{kiuchi-gamma}
\sum_{k\leq x}\frac{1}{k^s}  \sum_{j=1}^{k^s}s_{k}^{(s)}(j)\log\Gamma\left(\frac{j}{k^s}\right)
=\log\sqrt{2\pi} \sum_{d\ell\leq x}\frac{f(d)}{d^s}g(\ell)  \\
 - \log\sqrt{2\pi}\sum_{d\ell\leq x}\frac{f(d)}{d^s}\frac{g(\ell)}{\ell^s} 
          - \frac{s}{2} \sum_{d\ell\leq x}\frac{f(d)}{d^s}\frac{g(\ell)}{\ell^s}\log \ell,
\end{multline}
and
\begin{equation}                                                                
\label{Ber-BBB}  
\sum_{k\leq x}\frac{1}{k^s}\sum_{j=0}^{k^{s}-1} B_{m}\left(\frac{j}{k^s}\right)s_{k}^{(s)}(j)
=B_{m}\sum_{d \ell\leq x}\frac{f(d)}{d^s}\frac{g(\ell)}{\ell^{ms}},
\end{equation}
\\

For any real or complex number $a$, arithmetic functions $\phi_{a}$ and $\psi_{a}$ denote the Jordan totient function and the generalized Dedekind function
defined by $\id_{a}*\mu$ and $\id_{a}*|\mu|$, respectively where $\id_a(n)=n^a$ for $n\in \mathbb{N}_0$. The von Mangoldt function is denoted by $\Lambda$.\\
Now, if we denote sums on the left-hand side of \eqref{kiuchi-gamma1}, \eqref{kiuchi-gamma} and of \eqref{Ber-BBB} by $M_r^{(s)}(x;f,g)$, $A^{(s)}(x;f,g)$ and $H_{m}^{(s)}(x;f,g)$ respectively, take $f*\mu$ in place of $f$ and $g=\1$ and use the fact that $f*\mu*\1=f$, $f*\mu*\id_s=f*\phi_s$ and $f*\mu*\log=f*\Lambda$, we get 

\begin{multline}                                                                  \label{Kiuchi-Gamma1}
 M_r^{(s)}(x;f*\mu, \1) := \sum_{k\leq x}\frac{1}{k^{s(r+1)}}  \sum_{j=1}^{k^s}j^r
 \sum_{\substack{d|k \\ d^{s}|j}}f*\mu(d)
 \\
= \frac{1}{2}\sum_{n\leq x}\frac{f(n)}{n^s} 
+\frac{1}{r+1} \sum_{d\ell\leq x}\frac{f*\mu (d)}{d^s} 
  +\frac{1}{r+1}\sum_{m=1}^{[r/2]}\binom{r+1}{2m}B_{2m}\sum_{d\ell\leq x}\frac{f*\mu (d)}{d^s \ell^{2ms}},    
\end{multline} 

\begin{multline}                                                                   \label{Kiuchi-Gamma}
 A^{(s)}(x;f*\mu, \1) := \sum_{k\leq x}\frac{1}{k^s}  \sum_{j=1}^{k^s}\log\Gamma\left(\frac{j}{k^s}\right)
 \sum_{\substack{d|k \\ d^{s}|j}}f*\mu(d)
 \\
=\log\sqrt{2\pi}  \sum_{n\leq x}\frac{f*\phi_{s}(n)}{n^s} 
  - \log\sqrt{2\pi}\sum_{n\leq x}\frac{f(n)}{n^s} 
  - \frac{s}{2} \sum_{n\leq x}\frac{f*\Lambda(n)}{n^s},    
\end{multline}
and
\begin{equation*}                                                                             \label{Ber-GCD}
H_{m}^{(s)}(x;f*\mu, \1):=\sum_{k\leq x}\frac{1}{k^s}\sum_{j=0}^{k^{s}-1}   
                    B_{m}\sum_{\substack{d|k \\ d^{s}|j}} f*\mu(d)          
                 = B_{m}\sum_{d\ell\leq x}\frac{f*\mu(d)}{d^s}\frac{1}{\ell^{ms}}.
\end{equation*}
We will simply write $M_r^{(s)}(x;f)$, $A^{(s)}(x;f)$ and $H_{m}^{(s)}(x;f)$ instead of writing $M_r^{(s)}(x;f*\mu, \1)$, $A^{(s)}(x;f*\mu, \1)$ and $H_{m}^{(s)}(x;f*\mu, \1)$ respectively. 
When $m=1$ and  $2m$, we deduce that
\begin{align}                                                                 \label{Ber-GCD1}
&H_{1}^{(s)}(x;f) = -\frac12 \sum_{n\leq x}\frac{f(n)}{n^s}, \\                                                                  \label{Ber-GCD2m}
&H_{2m}^{(s)}(x;f) = B_{2m} \sum_{d\ell\leq x}\frac{f*\mu(d)}{d^s}\frac{1}{\ell^{2ms}},  
\end{align}
respectively. In~\cite{K2} and \cite{K3}, the authors studied the special case of $M_r^{(s)}(x;f)$, $A^{(s)}(x;f)$ and $H_{m}^{(s)}(x;f)$ when $s=1$. They gave several interesting asymptotic formulas for  weighted averages of $\gcd$-sum function $f(\gcd(k,j))$ with $f$ belonging to various multiplicative functions.  We recall that the $\gcd$-sum function, which is also known as Pillai's arithmetic function,  is essentially defined by 
$$P(n)=\sum_{k=1}^{n}\gcd(k,n).$$
Properties and various generalizations of $P(n)$ have been widely studied by many authors.
For a nice survey on this function, see \cite{To1}.  One of many generalizations of $P(n)$ is given by the identity
$$
P_{f}(n)=\sum_{k=1}^{n}f(\gcd(k,n))
$$
for an arbitrary arithmetical function $f$.\\

For the general case when $s>1$, the authors investigated, in \cite{K3}, the function $M_r^{(s)}(x;f)$ and gave asymptotic formulas of it with $f=\id_{s+a}, \phi_s, \psi_s, \phi_{s+a}, \psi_{s+a},$ $\tau*\id_s$, for any fixed real number $a$ such that $-1<a<0.$
The aim of this paper is to give asymptotic formulas for $A^{(s)}(x;f)$ and $H_1^{(s)}(x;f)$ and $H_{2m}^{(s)}(x;f)$ defined by identities \eqref{Kiuchi-Gamma}, \eqref{Ber-GCD1} and \eqref{Ber-GCD2m} with various multiplicative functions $f$.
Furthermore, we also consider several Dirichlet series associated with those functions. This work is a continuation of~\cite{K3}.


\section{Main results } 
\label{section2}
Before proceeding further we need to fix additional notations. Let $\tau$ and $\sigma$ be the divisor function and the sum of divisors function (sigma function) defined by $\1*\1$ and $\id*\1$, respectively. Let $\theta$ be the number appearing in the Dirichlet divisor problem 
\begin{equation}                                                                               \label{Dirichlet}
\sum_{n\leq x}\tau(n)  = x\log x + (2\gamma-1)x + \Delta(x),
\end{equation}
with $\Delta(x)=O\left(x^{\theta+\varepsilon}\right)$ for $\varepsilon>0$. 
The best estimate up to date for $\Delta(x)$ due to Huxley \cite{H} is 
$$
O\left(x^{131/416}(\log x)^{26947/8320}\right).$$ The constant term $\gamma$ is the Euler constant.  \\
More general, for any real or complex number $a$, the function $\sigma_{a}$  denotes  a generalized  divisor function where $\sigma_a=\id_{a}*\1$. 
We recall that   
\begin{equation}                                                                        \label{Delta-a}
\sum_{n\leq x}\sigma_{a}(n) = 
  \zeta(1-a)x +  \frac{\zeta(1+a)}{1+a} x^{1+a}  -  \frac{\zeta(-a)}{2} +  \Delta_{a}(x)    
\end{equation}
for  $-1 < a < 0$. Here $\zeta$ denotes the Riemann zeta-function. The problem of improving  $\Delta_{a}(x)$ is known as the generalized Dirichlet divisor problem. 
The classical estimate for  $\Delta_{a}(x)$ is 
\begin{equation}                                                                               \label{Peter}
\Ocal_{a}\left(x^{\frac{1+a}{3}+\varepsilon}\right),
\end{equation}
for any small  number $\varepsilon>0$, see~\cite{Pe}.

\subsection{The function $A^{(s)}(x;f)$} 
Among  many  possible  applications of $A^{(s)}(x;f)$,  we study in particular the following four cases when $f=\phi_s, \phi_{s+a}, \psi_s,$ and $\psi_{s+a}$ with $s$ being any fixed integer $s\geq 2$ and $-1<a<0$. Define the functions $D_{s}(x)$ and  ${\widetilde{D_{s}}}(x)$   by       
\begin{equation}                                                               \label{lem411}
D_{s}(x) = -\sum_{d\leq x}\frac{\mu(d)}{d^s}\vartheta\left(\frac{x}{d}\right) - \frac{1}{2\zeta(s)}
\end{equation}
and
\begin{equation}                                                          \label{lem421}
{\widetilde{D_{s}}}(x) = -\sum_{d\leq x}\frac{|\mu(d)|}{d^s}\vartheta\left(\frac{x}{d}\right) - \frac{\zeta(s)}{2\zeta(2s)}, 
\end{equation}
where  
$
\vartheta(x)=x - [x] - \frac{1}{2}.$
From  \eqref{Kiuchi-Gamma}, we have the following main results.  
\begin{theorem}  
\label{th8}
For any  real number $x >1$,   fixed  number $a$ such that $- 1 < a < 0$, 
and  fixed  integer $s\geq 2$,    we have
\begin{multline}                                                                
\label{sum_th81} 
A^{(s)}(x;\phi_{s})  = 
   \frac{\log\sqrt{2\pi}}{\zeta(s+1)^2}x\log x 
+ \frac{s\zeta'(s+1)}{2\zeta(s+1)^2} x                                    \\
+ \frac{\log\sqrt{2\pi}}{\zeta(s+1)^2}
\left(2\gamma -1 - 2\frac{\zeta'(s+1)}{\zeta(s+1)}  - \zeta(s+1) \right) x  + Y_1^{(s)}(x),          
\end{multline}
\begin{multline}  
\label{sum_th91} 
A^{(s)}(x;\phi_{s+a})  = 
   \frac{\log\sqrt{2\pi}\zeta(1-a)}{\zeta(s+1)^2} x  
+ \frac{s\zeta'(s+a+1)}{2(1+a)\zeta(s+a+1)^2}x^{1+a}  \\
 +  \frac{\log\sqrt{2\pi}}{1+a}\left(\frac{\zeta(1+a)}{\zeta(s+a+1)^2} 
- \frac{1}{\zeta(s+a+1)}\right) x^{1+a} + Y_2^{(s)}(x),                      
\end{multline}
where 
\begin{multline}
\label{sum_th82}
Y_1^{(s)}(x) 
= \log\sqrt{2\pi} \sum_{d\leq x}\frac{\mu*\mu(d)}{d^s}\Delta\left(\frac{x}{d}\right) 
 + \frac{s}{2} \sum_{d\leq x}\frac{\mu*\Lambda(d)}{d^s} \vartheta\left(\frac{x}{d}\right)                 \\
- \frac{s\zeta'(s)}{4\zeta(s)^{2}} - \log\sqrt{2\pi}~D_{s}(x)  +  \Ocal_{s}\left(x^{1-s}(\log x)^{2}\right),        
\end{multline}
and 
\begin{multline}
\label{sum_th92}
Y_2^{(s)}(x) 
= \log\sqrt{2\pi} \sum_{d\leq x}\frac{\mu*\mu(d)}{d^s}\Delta_{a}\left(\frac{x}{d}\right) 
 \\- \frac{\zeta(-a)}{2\zeta(s)^2}\log\sqrt{2\pi}                                           
- \frac{\zeta(-a)}{\zeta(s)}\left(\log\sqrt{2\pi}-\frac{s\zeta'(s)}{2\zeta(s)}\right) 
 +  \Ocal_{s,a}\left(x^{a}\right).                                
\end{multline}
\end{theorem}
\begin{theorem}  
\label{th10}
Under the hypothesis of Theorem~\ref{th8}, we have
\begin{multline}  
\label{sum_th101}
A^{(s)}(x;\psi_{s})  = 
   \frac{\log\sqrt{2\pi}}{\zeta(2s+2)}x\log x 
+ \frac{s\zeta'(s+1)}{2\zeta(2s+2)} x          \\
+ \frac{\log\sqrt{2\pi}}{\zeta(2s+2)}
\left(2\gamma -1 - 2\frac{\zeta'(2s+2)}{\zeta(2s+2)} - \zeta(s+1) \right) x  + Y_3^{(s)}(x),     
\end{multline}
and 
\begin{multline}  
 \label{sum_th111} 
A^{(s)}(x;\psi_{s+a}) = 
   \frac{\log\sqrt{2\pi}\zeta(1-a)}{\zeta(2s+2)} x  
+ \frac{s\zeta'(s+a+1)}{2(1+a)\zeta(2s+2a+2)}x^{1+a}       \\
 +  \frac{\log\sqrt{2\pi}}{1+a}\left(\frac{\zeta(1+a)}{\zeta(2s+2a+2)} 
- \frac{\zeta(s+a+1)}{\zeta(2s+2a+2)}\right) x^{1+a} 
 +  Y_4^{(s)}(x),                                
\end{multline}
where 
\begin{multline}
\label{sum_th102}
Y_3^{(s)}(x) 
 = \log\sqrt{2\pi} \sum_{d\leq x}\frac{\mu*|\mu|(d)}{d^s}\Delta\left(\frac{x}{d}\right) - \frac{s\zeta'(s)}{4\zeta(s)^{2}}     \\        
   + \frac{s}{2} \sum_{d\leq x}\frac{\mu*|\mu|(d)}{d^s} \vartheta\left(\frac{x}{d}\right) 
- \log\sqrt{2\pi}~{\widetilde{D_{s}}}(x) + \Ocal_{s}\left(x^{1-s}(\log x)^{2}\right),    
\end{multline}
and 
\begin{multline}
\label{sum_th112}
Y_4^{(s)}(x) 
= \log\sqrt{2\pi} \sum_{d\leq x}\frac{\mu*|\mu|(d)}{d^s}\Delta_{a}\left(\frac{x}{d}\right)
 - \frac{\zeta(-a)}{2\zeta(2s)}\log\sqrt{2\pi}   \\
- \frac{\zeta(-a)}{\zeta(s)}\left(\log\sqrt{2\pi}-\frac{s\zeta'(s)}{2\zeta(s)}\right) 
 +  \Ocal_{s,a}\left(x^{a}\right).                  
\end{multline}
\end{theorem}
\begin{remark}
Note that, by using elementary results of $\Delta$ and $\Delta_a$, it is easy to check that our functions $Y_1(x)$ and $Y_3(x)$ are estimated by $\Ocal_s\left( x^{\frac{1}{2}+\epsilon} \right)$, and $Y_2(x)$ and $Y_4(x)$ are estimated by $\Ocal_{s,a}\left(x^{\frac{1+a}{3}+\epsilon}\right).$ Hence we get 
$$
\lim_{x\to \infty}\frac{ A^{(s)}(x;\phi_{s})}{x\log x} =\frac{\log \sqrt{2\pi}}{\zeta(s+1)^2}, 
\quad 
\quad  
\lim_{x\to \infty}\frac{ A^{(s)}(x;\psi_{s})}{x\log x} =\frac{\log \sqrt{2\pi}}{\zeta(2s+2)},   
$$
and 
$$
\lim_{x\to \infty}\frac{ A^{(s)}(x;\phi_{s+a})}{x} =\frac{\log \sqrt{2\pi}\zeta(1-a)}{\zeta(s+1)^2}, 
\quad 
\quad  
\lim_{x\to \infty}\frac{ A^{(s)}(x;\psi_{s+a})}{x} =\frac{\log \sqrt{2\pi}\zeta(1-a)}{\zeta(2s+2)}.   
$$
From the above, we conclude that it is difficult to improve the $\Ocal$-terms in our theorems, since they basically depend on estimations of $\Delta$ and $\Delta_a$.  
\end{remark}

\subsection{Dirichlet series associated with $A^{(s)}(x; f).$} 


Let $F(w)$ and $G(w)$ be two functions represented by Dirichlet series as follows   
\begin{align*}
&F(w)=\sum_{k=1}^{\infty}\frac{f(k)}{k^w},   \qquad  {\rm Re} (w) > \sigma_{1}, 
\\
&G(w)=\sum_{k=1}^{\infty}\frac{g(k)}{k^w},  \qquad  {\rm Re} (w) > \sigma_{2},
\end{align*}
which are absolutely convergent in the half-plane ${\rm Re} (w) > \sigma_{1}$ and ${\rm Re} (w) > \sigma_{2}$,  respectively. 
Then Dirichlet series of the first derivative $G'(w)$ of $G(w)$ is given by 
$$
G'(w) = -\sum_{k=1}^{\infty}\frac{g(k)}{k^w}\log k, \qquad  {\rm Re} (w) > \sigma_{2}. 
$$ 

For $s, k \in \mathbb{N}_0$, define the weighted average $\kappa^{(s)}(k; f, g) $ by 
\begin{equation*}
\kappa^{(s)}(k; f, g)= \frac{1}{k^s}  \sum_{j=1}^{k^s}s_{k}^{(s)}(j)\log\Gamma\left(\frac{j}{k^s}\right)
\end{equation*}
In ~\cite{K2}, the first author showed that  
\begin{equation}                                                                \label{k-g}
\kappa^{(s)}(k; f, g)
=\log\sqrt{2\pi} \frac{f}{\id_{s}}*g (k)  
- \log\sqrt{2\pi}\frac{f*g(k)}{k^s}  - \frac{s}{2k^s} (f*g \log)(k).       
\end{equation}
Let $K^{(s)}(w; f, g)$ be a function represented by Dirichlet series as follows: 
$$
K^{(s)}(w; f, g)   := \sum_{k=1}^{\infty}\frac{\kappa^{(s)}(k; f, g) }{k^w},  
$$
where this summation is absolutely convergent for ${\rm Re} (w) >\alpha$. 
We use a property of Dirichlet series to obtain    
\begin{eqnarray*}                                                                      
K^{(s)}(w; f, g)  
&=& \log\sqrt{2\pi} \sum_{k=1}^{\infty}\sum_{\ell n=k}\frac{f(\ell)}{\ell^s}g(n)\frac{1}{k^w}   \\
& & - \log\sqrt{2\pi} \sum_{k=1}^{\infty}\sum_{\ell n=k} f(\ell) g(n)  \frac{1}{k^{w+s}} 
          - \frac{s}{2} \sum_{k=1}^{\infty}\sum_{\ell n=k} f(\ell) g(n)\log n  \frac{1}{k^{w+s}}   \\
&=& \log\sqrt{2\pi} \sum_{\ell=1}^{\infty} \frac{f(\ell)}{\ell^{w+s}}\sum_{n=1}^{\infty} \frac{g(n)}{n^w}   \\
&&- \log\sqrt{2\pi} \sum_{\ell=1}^{\infty} \frac{f(\ell)}{\ell^{w+s}}\sum_{n=1}^{\infty}\frac{g(n)}{n^{w+s}} 
 - \frac{s}{2}     \sum_{\ell=1}^{\infty} \frac{f(\ell)}{\ell^{w+s}}\sum_{n=1}^{\infty}\frac{g(n)\log n}{n^{w+s}}.                 
\end{eqnarray*}
This leads to the following result 
\begin{equation*}                                                            
  K^{(s)}(w; f, g)   
= \log\sqrt{2\pi} F(w+s) \left(G(w) -  G(w+s)\right) 
 +  \frac{s}{2}     F(w+s) G'(w+s). 
\end{equation*}
Taking $f*\mu$ in place of $f$ and $g=\1$, the above identity becomes 
\begin{equation*}                                                               
 K^{(s)}(w; f*\mu, \1)  
= \log\sqrt{2\pi} F(w+s)\left(\frac{\zeta(w)}{\zeta(w+s)} - 1\right)  
 +  \frac{s}{2} F(w+s) \frac{\zeta'(w+s)}{\zeta(w+s)}. 
\end{equation*}
Again, we write $K^{(s)}(w;f) $ instead of writing $K^{(s)}(w; f*\mu, \1).$ 
For $f=\phi_s, \psi_s, \phi_{s+a}$ and $ \psi_{s+a}$ successively, we deduce the following identities: 
\begin{corollary}                                                                
\label{cor1}
Under the above notations, we have 
\begin{align*}                                                               
&K^{(s)}(w; \phi_{s}*\mu)  
 =\log\sqrt{2\pi} \frac{\zeta(w)}{\zeta(w+s)} \left(\frac{\zeta(w)}{\zeta(w+s)} - 1\right)   
 +  \frac{s}{2} \frac{\zeta(w)\zeta'(w+s)}{\zeta(w+s)^2},  
\\                                             
 &K^{(s)}(w; \psi_{s}*\mu)  
 =\log\sqrt{2\pi} \frac{\zeta(w)\zeta(w+s)}{\zeta(2w+2s)} \left(\frac{\zeta(w)}{\zeta(w+s)} - 1\right)  
 +  \frac{s}{2} \frac{\zeta(w)\zeta'(w+s)}{\zeta(2w+2s)},          
\\
& K^{(s)}(w; \phi_{s+a}*\mu)  
 =\log\sqrt{2\pi} \frac{\zeta(w-a)}{\zeta(w+s)} \left(\frac{\zeta(w)}{\zeta(w+s)} - 1\right) 
 +  \frac{s}{2} \frac{\zeta(w-a)\zeta'(w+s)}{\zeta(w+s)^2}                   
\\
 & K^{(s)}(w; \psi_{s+a}*\mu)  
 =\log\sqrt{2\pi} \frac{\zeta(w-a)\zeta(w+s)}{\zeta(2w+2s)} \left(\frac{\zeta(w)}{\zeta(w+s)} - 1\right)  
   +  \frac{s}{2} \frac{\zeta(w-a)\zeta'(w+s)}{\zeta(2w+2s)}.    
\end{align*}
\end{corollary}
\subsection{ The function $H_{m}^{(s)}(x;f)$} 
Now, we consider the partial sums of the weighted average of $s_k^{(s)}(j)$ involving Bernoulli polynomials. We establish asymptotic formulas of \eqref{Ber-GCD1} and \eqref{Ber-GCD2m} 
for $f=\phi_s,$ $\phi_{s+a},$ $\psi_s,\psi_{s+a}$,  with $-1<a<0$. 
Our results are precisely the following: 
\begin{theorem}  \label{th12}
For any  real number $x >1$,   fixed  integers $s \geq 2$, $m\geq 1$,   we have
\begin{align} 
\label{sum_th121}
&H_{1}^{(s)}(x;\phi_{s}) = 
-\frac{1}{2\zeta(s+1)}x - \frac{D_{s}(x)}{2} + \Ocal_{s}\left(x^{1-s}\right),    \\
&
\label{sum_th122}  
H_{1}^{(s)}(x;\psi_{s}) = 
-\frac{\zeta(s+1)}{2\zeta(2s+2)}x - \frac{{\widetilde{D_{s}}}(x)}{2} + \Ocal_{s}\left(x^{1-s}\right),   
\end{align}
where $D_{s}(x)$ and ${\widetilde{D_{s}}}(x)$ are given by 
\eqref{lem411} and \eqref{lem421}.  
Moreover, we have 
\begin{align} 
 \label{sum_th123}                       
&H_{2m}^{(s)}(x;\phi_{s}) = B_{2m} \frac{\zeta(2ms+1)}{\zeta(s+1)^2}x  
- B_{2m} \frac{\zeta(2ms)}{2\zeta(s)^2}  + \Ocal_{m,s}\left(x^{1-s}\log x\right), \\
&
\label{sum_th124}
H_{2m}^{(s)}(x;\psi_{s}) = B_{2m} 
\frac{\zeta(2ms+1)}{\zeta(2s+2)}x  - B_{2m} \frac{\zeta(2ms)}{2\zeta(2s)}  + \Ocal_{m,s}\left(x^{1-s}\log x\right).        
\end{align}
\end{theorem}
\begin{theorem}  
\label{th13}
Under the hypothesis of Theorem~\ref{th12}. For any fixed  real number $a$  such  that  $-1 < a < 0$,  we have
\begin{equation}   
\label{sum_th131}                                      
H_{1}^{(s)}(x;\phi_{s+a})  =  
-\frac{1}{2(1+a)\zeta(s+a+1)} x^{1+a} - \frac{\zeta(-a)}{2\zeta(s)} + \Ocal_{s,a}\left(x^{a}\right),                  
\end{equation} 
\begin{equation} 
 \label{sum_th132}
H_{1}^{(s)}(x;\psi_{s+a})  =  
-\frac{\zeta(s+a+1)}{2(1+a)\zeta(2s+2a+2)}x^{1+a} - \frac{\zeta(-a)\zeta(s)}{2\zeta(2s)}   
 + \Ocal_{s,a}\left(x^{a}\right), 
\end{equation}
\begin{equation}
 \label{sum_th133}
H_{2m}^{(s)}(x;\phi_{s+a}) =  \frac{ B_{2m}\zeta(2ms+a+1)}{(1+a)\zeta(s+a+1)^2} x^{1+a}         
 +  \frac{ B_{2m}\zeta(-a)\zeta(2ms)}{\zeta(s)^{2}} + \Ocal_{a,m,s}\left(x^{a}\right),       
\end{equation}
and  
\begin{equation}
\label{sum_th134}
H_{2m}^{(s)}(x;\psi_{s+a}) = 
 \frac{B_{2m}\zeta(2ms+a+1)}{(1+a)\zeta(2s+2a+2)} x^{1+a}               
+ \frac{B_{2m}\zeta(-a)\zeta(2ms)}{\zeta(2s)} + \Ocal_{a,m,s}\left(x^{a}\right). 
\end{equation}
\end{theorem}
\subsection{ Dirichlet series associated with $H_{m}^{(s)}(x;f)$} 
From~\cite{K2}, we recall that
\begin{equation}                                                                 \label{Ber-AAA}                
\frac{1}{k^s}\sum_{j=0}^{k^{s}-1} B_{m}\left(\frac{j}{k^s}\right)s_{k}^{(s)}(j) 
=B_{m}\sum_{d\ell=k}\frac{f(d)}{d^s}\frac{g(\ell)}{\ell^{ms}}.  
\end{equation}
We introduce the notation $\nu_{m}^{(s)}(k;f, g)$ on the left-hand side of the above formula for convenience. Then, Dirichlet series associated with   $\nu_{f,g;m}^{(s)}(k)$ is given by 
\begin{align}                                                                    \label{Ber-MM}
\mathcal{L}_{m}^{(s)}(w; f,g) := \sum_{k=1}^{\infty}\frac{\nu_{m}^{(s)}(k;f, g)}{k^w},  
\end{align} 
which is absolutely  convergent for ${\rm Re} (w) >\beta$. 
We use identity \eqref{Ber-AAA} to proceed: 
\begin{equation*}                                                            
\mathcal{L}_{m}^{(s)}(w; f,g)  
= B_{m}  \sum_{k=1}^{\infty}\left(\sum_{\ell n=k}\frac{f(\ell)}{\ell^s}\frac{g(n)}{n^{ms}}\right)\frac{1}{k^w}  
= B_{m}  \sum_{\ell=1}^{\infty} \frac{f(\ell)}{\ell^{w+s}}\sum_{n=1}^{\infty} \frac{g(n)}{n^{w+ms}}.
\end{equation*}
This leads at once to 
\begin{equation}                                                                \label{sum_th876} 
\mathcal{L}_{m}^{(s)}(w; f,g) 
= B_{m}F(w+s) G(w+ms).   
\end{equation}
Now, we replace  $f*\mu$ by $f$ and $g=\1$ into \eqref{sum_th876}  to obtain 
\begin{equation*}                                                             
\mathcal{L}_{m}^{(s)}(w; f*\mu,\1) := \mathcal{L}_{m}^{(s)}(w;f) 
= B_{m} F(w+s) \frac{\zeta(w+ms)}{\zeta(w+s)}.    
\end{equation*}
Then, we deduce the following results:
\begin{corollary}                                                      \label{cor2}
Under the above notations, we have
\begin{equation*}                                                            
 \mathcal{L}_{m}^{(s)}(w;\phi_{s}*\mu)  
 =B_{m} \frac{\zeta(w)\zeta(w+ms)}{\zeta(w+s)^2},                          
 \quad 
  \mathcal{L}_{m}^{(s)}(w;\phi_{s+a}*\mu)  
 =B_{m}  \frac{\zeta(w-a)\zeta(w+ms)}{\zeta(w+s)^2},                               \end{equation*}
 \begin{equation*}
     \mathcal{L}_{m}^{(s)}(w;\psi_{s}*\mu)  
 =B_{m} \frac{\zeta(w)\zeta(w+ms)}{\zeta(2w+2s)}                         \quad                                                            
  \mathcal{L}_{m}^{(s)}(w;\psi_{s+a}*\mu)  
 =B_{m} \frac{\zeta(w-a)\zeta(w+ms)}{\zeta(2w+2s)}.    
\end{equation*}
\end{corollary}

\section{ Auxiliary results}

Before going into the proof of main results, we need to  give  some  auxiliary  lemmas. 
\begin{lemma} 
\label{lem4}
For any sufficiently large  number $x>1$  and  fixed  integer  $s\geq 2$, 
we have 
\begin{align}                                                                                 \label{lem41}                          
&\sum_{n\leq x}\frac{\phi_{s}(n)}{n^s}=
\frac{1}{\zeta(s+1)}x^{} + D_{s}(x) + \Ocal_{s}\left(x^{1-s}\right),     
\\                                                                
\label{lem42}
&\sum_{n\leq x}\frac{\psi_{s}(n)}{n^s} =
\frac{\zeta(s+1)}{\zeta(2s+2)} x^{} + {\widetilde{D_{s}}}(x) + \Ocal_{s}\left(x^{1-s}\right), 
\end{align}
where $D_{s}(x)$ and ${\widetilde{D_{s}}}(x)$ are defined by \eqref{lem411} and $\eqref{lem421}$.
\end{lemma}
\begin{proof}  
The formula \eqref{lem41} follows from (2.9) in \cite{K}. 
On the other hand, we have 
\begin{eqnarray*}
   \sum_{n\leq x}\frac{\psi_{s}(n)}{n^s} 
&=& \sum_{d\leq x}\frac{|\mu(d)|}{d^s}\sum_{\ell\leq x/d}1  \\
&=& x \sum_{d\leq x}\frac{|\mu(d)|}{d^{s+1}} 
 -   \sum_{d\leq x}\frac{|\mu(d)|}{d^{s}}\vartheta\left(\frac{x}{d}\right) 
 -   \frac12 \sum_{\ell\leq x}\frac{|\mu(\ell)|}{\ell^s}   \\
&=& \frac{\zeta(s+1)}{\zeta(2s+2)}x 
 - \sum_{d\leq x}\frac{|\mu(d)|}{d^{s}}\vartheta\left(\frac{x}{d}\right) 
 - \frac{\zeta(s)}{2\zeta(2s)} + \Ocal_{s}\left(x^{1-s}\right), 
\end{eqnarray*}
which completes the proof of~\eqref{lem42}.
\end{proof}
\begin{lemma} 
\label{lem5}
For any sufficiently large  number $x>1$,   
fixed  integer  $s\geq 2$  and  fixed  number $a$ such that $-1<a<0$, we have 
\begin{align}                                                                
 \label{lem51}
&\sum_{n\leq x}\frac{\phi_{s+a}(n)}{n^s} =
\frac{1}{(1+a)\zeta(s+a+1)}x^{1+a} + \frac{\zeta(-a)}{\zeta(s)} + \Ocal_{a,s}\left(x^{a}\right),                                   
\\ &  
\label{lem52}
\sum_{n\leq x}\frac{\psi_{s+a}(n)}{n^s} =
\frac{\zeta(s+a+1)}{(1+a)\zeta(2s+2a+2)}x^{1+a} + \frac{\zeta(-a)\zeta(s)}{\zeta(2s)} + \Ocal_{a,s}\left(x^{a}\right).            
\end{align}
\end{lemma}
\begin{proof}  
From see~\cite{Ap}, Theorem 3.2 (b), for $-1<a<0$, we have
\begin{equation}                                                                             
\label{lem1-apo} 
 \sum_{n\leq x}n^{a} = \frac{x^{1+a}}{1+a} + \zeta(-a) + \Ocal_{a}\left(x^{a}\right),  
\end{equation}
We use the above formula to get 
\begin{eqnarray*}
   \sum_{n\leq x}\frac{\phi_{s+a}(n)}{n^s} 
&=& \sum_{d\leq x}\frac{\mu(d)}{d^s}\sum_{\ell\leq x/d}\ell^{a}  \\
&=& \frac{x^{1+a}}{1+a}\sum_{d\leq x}\frac{\mu(d)}{d^{s+a+1}} 
 + \zeta(-a)\sum_{d\leq x}\frac{\mu(d)}{d^{s}}  
 + \Ocal_{a}\left(x^{a} \sum_{\ell\leq x}\frac{1}{\ell^{s+a}}\right)   \\
&=& \frac{1}{(1+a)\zeta(s+a+1)}x^{1+a} 
 + \frac{\zeta(-a)}{\zeta(s)} + \Ocal_{a,s}\left(x^{a}\right),   
\end{eqnarray*}
which completes the proof of \eqref{lem51}.  
Similarly we get \eqref{lem52}.   
\end{proof}
\section{Proofs} 
\subsection{Proof of Theorem \ref{th8}} 
For $f=\phi_{s}$, we use identity \eqref{Kiuchi-Gamma} to write

\begin{equation*}                                                             
 A^{(s)}(x;\phi_{s}) 
=\log\sqrt{2\pi}  \sum_{n\leq x}\frac{\phi_{s}*\phi_{s}(n)}{n^s} 
  - \log\sqrt{2\pi}\sum_{n\leq x}\frac{\phi_{s}(n)}{n^s} 
  - \frac{s}{2} \sum_{n\leq x}\frac{\phi_{s}*\Lambda(n)}{n^s}.    
\end{equation*}
To produce an asymptotic formula of $ A^{(s)}(x;\phi_{s})$, we need to estimate the first and  third sums above. Using the identities    
$$
\frac{\phi_{s}*\phi_{s}}{\id_s}
=\left(\frac{\mu}{\id_{s}}*\frac{\phi_{s}}{\id_{s}}\right)*\1  
= \frac{\mu*\mu}{\id_{s}}*\tau, 
$$
and \eqref{Dirichlet},  and formulas 
\begin{align}                                                                 \label{lem6-A}
&\sum_{d\leq x}\frac{\mu*\mu(d)}{d^{s+1}} = \frac{1}{\zeta(s+1)^2} + \Ocal_{s}\left(x^{-s} \log x \right),
\\ &                                                                                \label{lem6-B}
\sum_{d\leq x}\frac{\mu*\mu(d)}{d^{s+1}}\log d  = 2 \frac{\zeta'(s+1)}{\zeta(s+1)^{3}} + \Ocal_{s}\left(x^{-s} (\log x)^{2} \right),
\end{align}
we get  
\begin{eqnarray}                                                                                          \label{lem6-mu}
\sum_{n\leq x} \frac{\phi_{s}*\phi_{s}(n)}{n^s}  
&=&  \sum_{d\leq x}  \frac{\mu*\mu(d)}{d^s} \sum_{\ell\leq x/d} \tau(\ell)   \nonumber \\ 
&=& x\log x \sum_{d\leq x}\frac{\mu*\mu(d)}{d^{s+1}} 
-  x\sum_{d\leq x}\frac{\mu*\mu(d)}{d^{s+1}}\log d  \nonumber  \\ &&
+ (2\gamma -1) x \sum_{d\leq x}\frac{\mu*\mu(d)}{d^{s+1}} 
+  \sum_{d\leq x}\frac{\mu*\mu(d)}{d^s}\Delta_{}\left(\frac{x}{d}\right) \nonumber  \\
&=& \frac{1}{\zeta(s+1)^2}   x\log x 
+ \frac{1}{\zeta(s+1)^2}\left(2\gamma -1 -2 \frac{\zeta'(s+1)}{\zeta(s+1)}\right) x  \nonumber  \\
&&+  \sum_{d\leq x}\frac{\mu*\mu(d)}{d^s}\Delta_{}\left(\frac{x}{d}\right) + \Ocal_{s}\left(x^{1-s}(\log x)^2\right).    
\end{eqnarray}
For the third sum of  $A^{(s)}(x;\phi_{s})$, we use the fact that    
\begin{equation*}
\frac{\phi_{s}*\Lambda}{\id_{s}}  = 
\frac{\mu*\Lambda}{\id_{s}}*\1,
\quad  \text{and} \quad 
\sum_{n=1}^{\infty}\frac{\mu*\Lambda(n)}{n^{s+1}} = -\frac{\zeta'(s+1)}{\zeta(s+1)^2},
\end{equation*}
to write
\begin{eqnarray*}
\label{lem8-phi}
  \sum_{n\leq x} \frac{\phi_{s}*\Lambda(n)}{n^s} 
&=&  \sum_{d\leq x}  \frac{\mu*\Lambda(d)}{d^{s}} \sum_{\ell\leq x/d} 1   \\
&=& x \sum_{d\leq x}\frac{\mu*\Lambda(d)}{d^{s+1}} 
-    \sum_{d\leq x}\frac{\mu*\Lambda(d)}{d^{s}}\vartheta\left(\frac{x}{d}\right)  
-  \frac{1}{2} \sum_{d\leq x}\frac{\mu*\Lambda(d)}{d^s}   \\ 
&=& - \frac{\zeta'(s+1)}{\zeta(s+1)^2} x 
 -   \sum_{d\leq x}\frac{\mu*\Lambda(d)}{d^s} \vartheta\left(\frac{x}{d}\right) 
+  \frac{\zeta'(s)}{2\zeta(s)^2}  +  \Ocal_{s}\left(\frac{(\log x)^{2}}{x^{s-1}}\right).     
\end{eqnarray*}
From the latter formula, identity \eqref{lem41}, and \eqref{lem8-phi}, we get the identity \eqref{sum_th81}. \\
Now, we recall that 
\begin{equation*}                                                              
 A^{(s)}(x;\phi_{s+a}) 
=\log\sqrt{2\pi}  \sum_{n\leq x}\frac{\phi_{s}*\phi_{s+a}(n)}{n^s} 
 - \log\sqrt{2\pi}\sum_{n\leq x}\frac{\phi_{s+a}(n)}{n^s} 
  - \frac{s}{2} \sum_{n\leq x}\frac{\phi_{s+a}*\Lambda(n)}{n^s}.   
\end{equation*}
In order to prove the identity~\eqref{sum_th91}, we calculate each of three sums on the right-hand side above separately. For the first sum,
we use \eqref{Delta-a},  \eqref{lem6-A} and 
$$
\frac{\phi_{s+a}*\phi_{s}}{\id_s}=\frac{\phi_{s+a}*\mu}{\id_{s}}*\1 = 
\frac{\mu*\mu}{\id_{s}}*\sigma_{a} 
$$
to obtain    
\begin{eqnarray}                                                                 \label{lem7-mu} 
\sum_{n\leq x} \frac{\phi_{s+a}*\phi_{s}(n)}{n^s}   
&=&  \sum_{d\leq x}  \frac{\mu*\mu(d)}{d^{s}} \sum_{\ell\leq x/d} \sigma_{a}(\ell)  \nonumber  \\
&=& \zeta(1-a)x \sum_{d\leq x}\frac{\mu*\mu(d)}{d^{s+1}} 
+  \frac{\zeta(1+a)}{1+a} x^{1+a} \sum_{d\leq x}\frac{\mu*\mu(d)}{d^{s+a+1}}  \nonumber \\
 && -  \frac{\zeta(-a)}{2} \sum_{d\leq x}\frac{\mu*\mu(d)}{d^s} 
+  \sum_{d\leq x}\frac{\mu*\mu(d)}{d^s}\Delta_{a}\left(\frac{x}{d}\right)    \nonumber \\
&=& \frac{\zeta(1-a)}{\zeta(s+1)^2} x 
+  \frac{\zeta(1+a)}{(1+a)\zeta(s+a+1)^2} x^{1+a} 
-  \frac{\zeta(-a)}{2\zeta(s)^2}                                     \nonumber  \\
&& +  \sum_{d\leq x}\frac{\mu*\mu(d)}{d^s}\Delta_{a}\left(\frac{x}{d}\right)
+  \Ocal_{s,a}\left(x^{1-s}\log x\right). 
\end{eqnarray}
For the third sum of $A^{(s)}(x;\phi_{s+a})$, we use \eqref{lem1-apo} to obtain the formula
\begin{eqnarray*}                                                               \label{lem8-phia}
\sum_{n\leq x} \frac{\phi_{s+a}*\Lambda(n)}{n^s} 
&=&   \sum_{d\leq x}\frac{\mu*\Lambda(d)}{d^s}\sum_{\ell\leq x/d}\ell^{a}   \\ 
&=& -\frac{\zeta'(s+a+1)}{(1+a)\zeta(s+a+1)^2} x^{1+a} 
-   \frac{\zeta(-a)\zeta'(s)}{\zeta(s)^{2}}  + \Ocal_{s,a}\left(x^{a}\right). 
\end{eqnarray*}
Combining this latter, \eqref{lem51} and \eqref{lem8-phia}, we get \eqref{sum_th91}.
\subsection{Proof of  Theorem \ref{th10}} 
Taking $f=\psi_{s}$  into \eqref{Kiuchi-Gamma}, we have 
 \begin{equation*}                                                               A^{(s)}(x;\psi_{s}) 
=\log\sqrt{2\pi}  \sum_{n\leq x}\frac{\psi_{s}*\phi_{s}(n)}{n^s} 
 - \log\sqrt{2\pi}\sum_{n\leq x}\frac{\psi_{s}(n)}{n^s} 
   - \frac{s}{2} \sum_{n\leq x}\frac{\psi_{s}*\Lambda(n)}{n^s}.  
\end{equation*}
We notice that the first sum on the right-hand side of the above is rewritten as    
$$
\frac{\psi_{s}*\phi_{s}}{\id_s}=\left(\frac{\psi_{s}}{\id_{s}}*\frac{\mu}{\id_{s}}\right)*\1  
= \frac{\mu*|\mu|}{\id_{s}}*\tau. 
$$
Using \eqref{Dirichlet} and formulas 
\begin{align}                                                             
\label{lem6-C}
&\sum_{d\leq x}\frac{\mu*|\mu|(d)}{d^{s+1}} = \frac{1}{\zeta(2s+2)} + \Ocal_{s}\left(x^{-s} \log x \right),
\\&  \nonumber                                                             
\sum_{d\leq x}\frac{\mu*|\mu|(d)}{d^{s+1}}\log d  = 2 \frac{\zeta'(2s+2)}{\zeta(2s+2)^{2}} + \Ocal_{s}\left(x^{-s} (\log x)^{2} \right),
\end{align}
we infer    
\begin{eqnarray}                                                                \label{lem6-psi} 
 \sum_{n\leq x} \frac{\psi_{s}*\phi_{s}(n)}{n^s}  
&=&  \sum_{d\leq x}  \frac{\mu*|\mu|(d)}{d^s} \sum_{\ell\leq x/d} \tau(\ell)   \nonumber \\
&=& x\log x \sum_{d\leq x}\frac{\mu*|\mu|(d)}{d^{s+1}} 
-  x\sum_{d\leq x}\frac{\mu*|\mu|(d)}{d^{s+1}}\log d                     \nonumber  \\
&& + (2\gamma -1) x \sum_{d\leq x}\frac{\mu*|\mu|(d)}{d^{s+1}} 
+  \sum_{d\leq x}\frac{\mu*|\mu|(d)}{d^s}\Delta_{}\left(\frac{x}{d}\right)       \nonumber \\ 
&=& \frac{1}{\zeta(2s+2)}  x\log x 
+ \frac{1}{\zeta(2s+2)}\left(2\gamma -1 -2 \frac{\zeta'(2s+2)}{\zeta(2s+2)}\right) x   \nonumber  \\
&& +  \sum_{d\leq x}\frac{\mu*|\mu|(d)}{d^s}\Delta_{}\left(\frac{x}{d}\right) + \Ocal_{s}\left(x^{1-s}(\log x)^2\right).    
\end{eqnarray}
For the third sum of $ A^{(s)}(x;\psi_{s}) $, we use the fact that  
$$
\frac{\psi_{s}*\Lambda}{\id_{s}}  = 
\frac{|\mu|*\Lambda}{\id_{s}}*\1 
\quad {\rm and}  \quad 
\sum_{n=1}^{\infty}\frac{|\mu|*\Lambda(n)}{n^{s+1}} = -\frac{\zeta'(s+1)}{\zeta(2s+2)}
$$
to obtain   
\begin{align}                                                                   \label{lem8-psi}
&  \sum_{n\leq x} \frac{\psi_{s}*\Lambda(n)}{n^s} 
=  \sum_{d\leq x}  \frac{|\mu|*\Lambda(d)}{d^{s}} \sum_{\ell\leq x/d} 1  \nonumber  \\
&= x \sum_{d\leq x}\frac{|\mu|*\Lambda(d)}{d^{s+1}} 
-    \sum_{d\leq x}\frac{|\mu|*\Lambda(d)}{d^{s}}\vartheta\left(\frac{x}{d}\right)  
-  \frac{1}{2} \sum_{d\leq x}\frac{|\mu|*\Lambda(d)}{d^s} \nonumber \\ 
&= - \frac{\zeta'(s+1)}{\zeta(2s+2)} x 
 -   \sum_{d\leq x}\frac{|\mu|*\Lambda(d)}{d^s} \vartheta\left(\frac{x}{d}\right) +  \frac{\zeta'(s)}{2\zeta(2s)} 
+  \Ocal_{s}\left(x^{1-s}(\log x)^{2}\right). 
\end{align}
From \eqref{lem6-psi}, \eqref{lem42} and \eqref{lem8-psi}, we get \eqref{sum_th101}.\\

In order to prove the identity~\eqref{sum_th111}, we are going to estimate each term on the right-hand side of the following formula 
\begin{equation*}                                                              
 A^{(s)}(x;\psi_{s+a}) 
=\log\sqrt{2\pi}  \sum_{n\leq x}\frac{\psi_{s+a}*\phi_{s}(n)}{n^s} 
 - \log\sqrt{2\pi}\sum_{n\leq x}\frac{\psi_{s+a}(n)}{n^s} 
  - \frac{s}{2} \sum_{n\leq x}\frac{\psi_{s+a}*\Lambda(n)}{n^s}.  
\end{equation*}
Since  the identity 
$$
\frac{\phi_{s}*\psi_{s+a}}{\id_s}=\left(\frac{\mu}{\id_{s}}*\frac{\psi_{s+a}}{\id_{s}}\right)*\1  
= \frac{\mu*|\mu|}{\id_{s}}*\sigma_a,  
$$
and the treatment of the first sum above is similar to that used in the proof of Theorem~\ref{th8}. This yields 
\begin{multline}
\label{lem7-psi}
 \sum_{n\leq x} \frac{\psi_{s+a}*\phi_{s}(n)}{n^s}  
= \frac{\zeta(1-a)}{\zeta(2s+2)}x 
+  \frac{\zeta(1+a)}{(1+a)\zeta(2s+2a+2)} x^{1+a}                                            
-  \frac{\zeta(-a)}{2\zeta(2s)}  \\+  \sum_{d\leq x}\frac{\mu*|\mu|(d)}{d^s}\Delta_{a}\left(\frac{x}{d}\right)
+  \Ocal_{s,a}\left(x^{1-s}\log x\right),    
\end{multline}
where we used \eqref{lem6-C} instead of \eqref{lem6-A}.
Similarly, the third sum is 
\begin{equation*}
\label{lem8-psia} 
\sum_{n\leq x} \frac{\psi_{s+a}*\Lambda(n)}{n^s}   =
- \frac{\zeta'(s+a+1)}{(1+a)\zeta(2s+2a+2)} x^{1+a} - \frac{\zeta(-a)\zeta'(s)}{\zeta(2s)}  
+ \Ocal_{s,a}\left(x^{a}\right).  
\end{equation*}
From the above and \eqref{lem52}, the proof is complete. 
\subsection{Proof of Theorems \ref{th12} and \ref{th13}}
By \eqref{Ber-GCD1} with replaced  $f$ by $\phi_s, \psi_s, \phi_{s+a}, \psi_{s+a}$ successively, 
and using \eqref{lem41}, \eqref{lem42}, \eqref{lem51} and \eqref{lem52}. 
We get at once \eqref{sum_th121}, \eqref{sum_th122}, \eqref{sum_th131} and \eqref{sum_th132}. 
To complete the proof of Theorems~\ref{th12} and \ref{th13}, it remains to prove the following relations 
\begin{equation}                                                                
\label{lem6-l2m} 
\sum_{d\ell\leq x}\frac{\phi_{s}*\mu(d)}{d^s}  \frac{1}{\ell^{2ms}} =
\frac{\zeta(2ms+1)}{\zeta(s+1)^2}x  - \frac{\zeta(2ms)}{2\zeta(s)^2}  + \Ocal_{m,s}\left(x^{1-s}\log x\right),                  
\end{equation}

\begin{equation}
\label{lem6-psi2m}
\sum_{d\ell\leq x}\frac{\psi_{s}*\mu(d)}{d^s}  \frac{1}{\ell^{2ms}} =
\frac{\zeta(2ms+1)}{\zeta(2s+2)}x  - \frac{\zeta(2ms)}{2\zeta(2s)}  + \Ocal_{m,s}\left(x^{1-s}\log x\right),                      
\end{equation}

\begin{equation}
\label{lem7-l2m}
\sum_{d\ell\leq x} \frac{\phi_{s+a}*\mu(d)}{d^s}  \frac{1}{\ell^{2ms}} =
\frac{\zeta(2ms+a+1)}{(1+a)\zeta(s+a+1)^2} x^{1+a}                               
 + \frac{\zeta(-a)\zeta(2ms)}{\zeta(s)^{2}} + \Ocal_{s,a,m}\left(x^{a}\right),   
\end{equation}
and 
\begin{equation}
\label{lem7-psi2m} 
\sum_{d\ell\leq x} \frac{\psi_{s+a}*\mu(d)}{d^s}  \frac{1}{\ell^{2ms}} =
\frac{\zeta(2ms+a+1)}{(1+a)\zeta(2s+2a+2)} x^{1+a}                              
+ \frac{\zeta(-a)\zeta(2ms)}{\zeta(2s)} + \Ocal_{s,a,m}\left(x^{a}\right).       
\end{equation}
We start with \eqref{lem6-l2m}. Using  the identity
\begin{equation*}
\frac{\phi_{s}*\mu}{\id_{s}}*\frac{1}{\id_{2ms}} = \frac{\mu*\mu}{\id_{s}}*\sigma_{-2ms},  
\end{equation*}
the formula    
\begin{equation}                                                                  \label{lem6-AA}
\sum_{\ell\leq x}\sigma_{-2ms}(\ell)  = 
\zeta(2ms+1) x -\frac{\zeta(2ms)}{2}  + \Ocal_{m,s}\left(x^{1-2ms}\right)
\end{equation}
and \eqref{lem6-A}, we find that  
\begin{eqnarray*}
 \sum_{d\ell\leq x}\frac{\phi_{s}*\mu(d)}{d^s}\frac{1}{\ell^{2ms}} &=& 
  \sum_{d\leq x}\frac{\mu*\mu(d)}{d^s}\sum_{\ell\leq x/d}\sigma_{-2ms}(\ell)  \\
&=&\zeta(2ms+1) x \sum_{d\leq x}\frac{(\mu*\mu)(d)}{d^{s+1}} 
- \frac{\zeta(2ms)}{2} \sum_{d\leq x}\frac{\mu*\mu(d)}{d^s}   \\ 
&& +  O\left(x^{1-2ms} \sum_{d\leq x}\frac{\tau(d)}{d^{s+1-2ms}}\right) \\
&=&\frac{\zeta(2ms+1)}{\zeta(s+1)^2}x  - \frac{\zeta(2ms)}{2\zeta(s)^2}  + \Ocal_{s,m}\left(x^{1-s}\log x\right), 
\end{eqnarray*}
as required.
By a similar argument to the above, by using \eqref{lem6-C} instead of \eqref{lem6-A}, we obtain \eqref{lem6-psi2m}.   
For \eqref{lem7-l2m}, using the identity
$$
\frac{\phi_{s+a}*\mu}{\id_{s}}*\frac{1}{\id_{2ms}} = 
\frac{\mu*\mu}{\id_{s}}*\frac{\sigma_{a+2ms}}{\id_{2ms}}, 
$$
and  the formula   
\begin{equation}                                                                              \label{lem7-B}
\sum_{\ell\leq x}\frac{\sigma_{a+2ms}(\ell)}{\ell^{2ms}} = 
\frac{\zeta(a+2ms+1)}{a+1}x^{a+1} + \zeta(-a)\zeta(2ms) + \Ocal_{s,a,m},\left(x^{a}\right),
\end{equation}
we get  
\begin{eqnarray*}
\sum_{d\ell\leq x}\frac{\phi_{s+a}*\mu(d)}{d^s}\frac{1}{\ell^{2ms}} 
&=& \sum_{d\ell\leq x}\frac{\mu*\mu(d)}{d^s}\frac{\sigma_{a+2ms}(\ell)}{\ell^{2ms}}  \\ 
&=&\frac{\zeta(2ms+a+1)}{(1+a)\zeta(s+a+1)^2} x^{1+a} 
+ \frac{\zeta(-a)\zeta(2ms)}{\zeta(s)^{2}} + \Ocal_{s,a,m}\left(x^{a}\right).  
\end{eqnarray*}
Similarly, we obtain \eqref{lem7-psi2m}.  
Which completes the proof of our theorems. 
\paragraph*{Acknowledgment.}
The second author is supported by the Austrian Science
Fund (FWF) : Projects F5507-N26, and F5505-N26 which are parts
of the special Research Program  `` Quasi Monte
Carlo Methods : Theory and Application''. Part of this work was done while the second author was supported by the Japan Society for the Promotion of Science (JSPS).  ``Overseas researcher under Postdoctoral Fellowship of JSPS".

\medskip\noindent {\footnotesize Isao Kiuchi: Department of Mathematical Sciences, Faculty of Science,
Yamaguchi University, Yoshida 1677-1, Yamaguchi 753-8512, Japan. \\
e-mail: {\tt kiuchi@yamaguchi-u.ac.jp}}

\medskip\noindent {\footnotesize Sumaia Saad Eddin: Institute of Financial Mathematics and Applied Number Theory, Johannes Kepler University, Altenbergerstrasse 69, 4040 Linz, Austria.\\
e-mail: {\tt sumaia.saad\_eddin@jku.at}

\begin{thebibliography}{99}

\bibitem{AA} D. R.  Anderson and T. M. Apostol, The evaluation of Ramanujan's sum and generalizations, {\it Duke Math. J.} {\bf 20}  (1952),  211--216.

\bibitem{Ap} T. M. Apostol,  {\it Introduction to  Analytic Number Theory},  Springer, 1976.

\bibitem{C1} E. Cohen,  An extension of Ramanujan's sums, {\it Duke Math. J.} {\bf 16} (1949), 85--90.

\bibitem{C2} E. Cohen, An extension of Ramanujan's sums II.  Additive Properties, {\it Duke Math. J.} {\bf 22} (1955), 543--559.

\bibitem{C3} E. Cohen, An extension of Ramanujan's sums III. Connections with totient functions, {\it Duke Math. J.} {\bf 23} (1956), 623--630.

\bibitem{H}  M. N. Huxley, Exponential sums and lattice points III, 
{\it  Proc. London Math. Soc.} {\bf 87} (2003),  591--609.


\bibitem{K} I. Kiuchi,  On sums of averages of generalized Ramanujan sums, 
{\it Tokyo  J.  Math.} {\bf 40} (2017),  255--275.   

\bibitem{K2} I. Kiuchi, Sums of averages of  generalized Ramanujan sums,    
 {\it J. Number Theory} {\bf 180} (2017),  310--348.  

\bibitem{K3} I. Kiuchi and S. Saad Eddin, Sums of weighted averages of gcd-sum functions, {\it Int. J. Number Theory} {\bf 14} (2018), 2699-2728.  

\bibitem{Pe} Y.-F. S. P\'{e}termann, Divisor problems and exponent pairs,     
 {\it Arch.  Math.} (Basel) {\bf 50} (1988), 243--250.

\bibitem{To1} L. T\'{o}th, A  survey of $\gcd$-sum functions,  
{\it J. Integer  Sequences} {\bf 13}  (2010), Article  10.8.1.
\end{thebibliography}
     \end{document}